\documentclass[a4paper]{article}
\usepackage[all]{xy}\usepackage[latin1]{inputenc}        
\usepackage[dvips]{graphics,graphicx}
\usepackage{amsfonts,amssymb,amsmath,color,mathrsfs, amstext}
\usepackage{amsbsy, amsopn, amscd, amsxtra, amsthm,authblk}
\usepackage{enumerate,algorithmic,algorithm}
\usepackage{upref}
\usepackage{geometry}
\usepackage[displaymath]{lineno}
\usepackage{float}
\usepackage{yhmath}
\usepackage{booktabs}
\usepackage{subcaption}
\usepackage{multirow}
\usepackage{makecell}
\usepackage[colorlinks,
            linkcolor=red,
            anchorcolor=red,
            citecolor=red
            ]{hyperref}

\numberwithin{equation}{section}

\def\e{\epsilon}

\def\R{\mathbb{R}}

\def\cD{\mathcal{D}}

\def\cA{\mathcal{A}}

\def\cX{\mathcal{X}}

\def\pc{\bar{*}}

\newcommand{\tcb}[1]{\textcolor{blue}{#1}}

\newtheorem{theorem}{Theorem}[section]
\newtheorem{lemma}{Lemma}[section]
\newtheorem{proposition}{Proposition}[section]
\newtheorem{remark}{Remark}[section]

\newtheorem{definition}{Definition}[section]

\makeatletter
\@mparswitchfalse
\makeatother

\title{On the completely positive kernels for nonuniform meshes}

\author[a]{Yuanyuan Feng\thanks{
E-mail: yyfeng@math.ecnu.edu.cn}}
\author[b]{Lei Li\thanks{E-mail: leili2010@sjtu.edu.cn}}

\affil[a]{School of  Mathematical Sciences, Shanghai Key Laboratory of PMMP, East China Normal University, Shanghai, 200241, P.R. China. }
\affil[b]{School of Mathematical Sciences, Institute of Natural Sciences, MOE-LSC, Shanghai Jiao Tong University, Shanghai, 200240, P.R.China.}

\date{}

\begin{document}

\maketitle

\begin{abstract}
The complete positivity, i.e., positivity of the resolvent kernels, for convolutional kernels is an important property for the positivity property and asymptotic behaviors of Volterra equations.  We inverstigate the discrete analogue of the complete positivity properties, especially for convolutional kernels on nonuniform meshes. Through an operation which we call pseudo-convolution, we introduce the complete positivity property for discrete kernels on nonuniform meshes and establish the criterion for the complete positivity. Lastly, we apply our theory to the L1 discretization of time fractional differential equations on nonuniform meshes.
\end{abstract}

\section{Introduction}

The time-delay memory is ubiquitous in physical models, which may be resulted from dimension reduction as in the generalized Langevin model for particles in heat bath (\cite{zwanzig73,zwanzig01,kouxie04,li2017fractional}) or may be resulted from viscoelasticity in soft matter (\cite{colemannoll1961,pd97}), or dielectric susceptibility for polarization \cite{stenzel05,cai13}, to name a few examples.
A basic model for the memory is the Volterra integral equations (see \cite{gripenberg1990volterra,miller1971smoothness,weis1975asymptotic,loy2014interconversion}). 
Let $\cX$ be a Banach space and $f: [0,\infty)\times \cX\to \cX$ be a given smooth function. The integral equation we consider in this work is
\begin{gather}\label{eq:vol}
u(t)=h(t)+\int_0^t a(t-s) f(s, u(s))\,ds,
\end{gather}
where $u: [0, T)\to \cX$ is the solution curve. Here, $a: [0,\infty)\to \R$ is the memory kernel.

Recall the standard one-sided convolution for two functions $u: [0,\infty)\to \R$ and $v: [0,\infty)\to \R$:
\begin{gather}
u*v(t)=\int_{[0, t]}u(s)v(t-s)\,ds.
\end{gather}
Such a convolution can be generalized to distributions whose supports
are on $[0,\infty)$ (see \cite[sections 2.1,2.2]{liliu18frac}). This convolution is commutative, associative. The identity is the Dirac delta $\delta$, defined by
\begin{gather}
\langle \delta, \varphi(\cdot)\rangle=\varphi(0), \forall \varphi\in C_c^{\infty}.
\end{gather}
With the convolution introduced, the Volterra integral equation \eqref{eq:vol} is then written as
\[
u(t)=h(t)+a*f(\cdot, u(\cdot)).
\]

Associated with the memory kernel $a$, the resolvent kernels considered in \cite{clement1979abstract,clement1981asymptotic,miller1968volterra} are crucial for inverstigating the properties of the equation, which are defined as follows.
\begin{definition}\label{def:resol}
Let $\lambda>0$. The resolvent kernels $r_{\lambda}$ and $s_{\lambda}$ for $a$ are defined respectively by
\begin{gather}\label{eq:resolformula}
\begin{split}
& r_{\lambda}+\lambda r_{\lambda}*a=\lambda a,\\
& s_{\lambda}+\lambda s_{\lambda}* a=1.
\end{split}
\end{gather}
\end{definition}
In \cite{clement1981asymptotic}, the complete positivity of the kernel $a$ is characterized by the nonnegativity of the resolvents $r_{\lambda}$ and $s_{\lambda}$. This is important for studying the positivity property and asymptotic behaviors of the solutions.

At the discrete level, it is desired that the complete positivity can be preserved. 
Besides, due to the memory kernels, especially some weakly singular kernels, the models often exhibit multi-scale behaviors \cite{cuesta2006convolution,tang2019energy,zhan2019complete}, which bring numerical challenge. The adaptive time-stepping is often adopted to address this issue \cite{mclean1996discretization,kopteva2019error,liao2019discrete,stynes2017error,li2019linearized,
 lyu2022symmetric}.  
 
Suppose that the computational time interval is $[0, T]$. Let $0=t_0<t_1<t_2<\cdots<t_N=T$ be the grid points.  We define
\begin{gather}
\tau_n:=t_n-t_{n-1}, \quad n\ge 1.
\end{gather}
Let $u_n$ be the numerical solution at $t_n$. By implicit discretization of the Volterra integral equation \eqref{eq:vol}, one may obtain
\begin{gather}\label{eq:integraldis}
u_n=h(t_n)+\sum_{j=1}^n a_{n-j}^n f(t_j, u_j).
\end{gather}
Here, $a_{n-j}^n$ is like the inegral of $a(t_n-s)$ on the interval $[t_{j-1}, t_j]$. 
 For the uniform meshes, the right hand side is the usual convolution for sequences and the concept of complete positivity is relatively easy to generalize. However, it remains open how this can be generalized to nonuniform meshes.
 
 In this work, we aim to address this question. In section \ref{sec:prelim}, we review the definition of the complete positivity
 and perform relevant discussions. In section \ref{sec:cpuniform}, we consider the complete positivity on uniform meshes.
 In section \ref{sec:pc}, we consider the pseudo-convolution which will then be used to study the complete positivity on nonuniform meshes  in section \ref{sec:cpnonuni}. Lastly in section \ref{sec:application}, we look at one illustrating example to see how our theory can be applied.

\section{The completely positive kernels}\label{sec:prelim}

In this section, we introduce some preliminaries and foundations for the discussion of this paper. 

\subsection{The complete positivity}\label{subsec:timecontinuous}

The resolvent kernels are useful to investigate the positivity and asymptotic properties of the Volterra type integral equations. The resolvent kernel $r_{\lambda}$ defined in \eqref{eq:resolformula} in fact satisfies
\begin{gather}\label{eq:interres}
(\delta +\lambda a)*(\delta-r_{\lambda})=\delta.
\end{gather}
It is also clearly that (see \cite{clement1981asymptotic})
\begin{gather}
s_{\lambda}=1*(\delta -r_{\lambda})=1-\int_0^t r_{\lambda}(\tau)\,d\tau.
\end{gather}
Formally, by the definition of $r_{\lambda}$, one has $\delta-r_{\lambda}=\lambda^{-1}r_{\lambda}*a^{(-1)}$, though the existence of the convolutional inverse $a^{(-1)}$ is not clear at this point. Noting that the complementary kernel (see Lemma \ref{lmm:contcp} below) $a^c=a^{(-1)}*1$, one finds that $s_{\lambda}=\lambda^{-1}r_{\lambda}*a^c$.

In \cite{clement1981asymptotic},  the so-called ``completely positive'' kernels were considered by Clement and Nohel.
\begin{definition}\label{def:cpcont}
Let $T>0$. A kernel $a\in L^1(0, T)$ is said to be completely positive if both the resolvent kernels $r_{\lambda}$ and $s_{\lambda}$ defined in Definition \ref{def:resol} are nonnegative for every $\lambda>0$.
\end{definition}

A sufficient condition is the following (see \cite{miller1968volterra}).
\begin{lemma}\label{lmm:continouslogconv}
If the kernel $a\in L^1(0, T)$ is nonnegative, nonincreasing and $t\mapsto \log a(t)$ is convex, then $a$ is completely positive.
\end{lemma}
In fact, the statement for the log-convexity of $a$ in  \cite{miller1968volterra} is that $t\mapsto a(t)/a(t+T)$ is nonincreasing for all $T>0$.

The following description of the complete positivity has been proved in \cite[Theorem 2.2]{clement1981asymptotic}. (The second claim has been mentioned in Remark (i) below the main result there.)
\begin{lemma}\label{lmm:contcp}
Let $T>0$. A kernel $a\in L^1(0, T)$ with $a\not\equiv 0$ is completely positive on $[0, T]$ if and only if there exists $\alpha\ge 0$ and $c \in L^1(0, T)$ nonnegative and nonincreasing  satisfying
\begin{gather}\label{eq:cpchar}
\alpha a+c*a=a*(\alpha \delta+c)=1_{t\ge 0}.
\end{gather}
Moreover, provided that $a$ is completely positive, $\alpha>0$ if and only if $a\in L^{\infty}(0, T)$ and in this case $a$ is in fact absolutely continuous on $[0, T]$.
\end{lemma}
 This result tells us that there is a complementary kernel $a^c=\alpha \delta+c$ for $a$. Cearly, $a^c$ is a  nonnegative and nonincreasing measure on $[0,T]$. It is nonincreasing in the sense that $a^c[t_0, t_0+\Delta t)\ge a^c[t_0+\Delta t, t_0+2\Delta t)$ for any $t_0\ge 0, \Delta t>0$. Here, $a^c[I]$ means the integral of the measure \tcb{$a^c$} on the interval $I$.

\subsection{Resolvents for completely monotone kernels}\label{subsec:cmresolvent}

In this subsection, we consider a special case, namely when the kernel is a completely monotone function \cite{widder41,ssv12}. 
A function $a: (0,\infty)\to \R$ is called completely monotone (CM) if $(-1)^na^{(n)}(t)\ge 0$ for all $n=0,1,2,\cdots$ and $t>0$.
It is known already that the completely monotone kernels are log-convex and thus completely positive by Bernstein theorem (see  \cite{clement1981asymptotic} and \cite{miller1968volterra}). Here, we show that the resolvent kernels are also completely monotone.

Motivated by a discrete analogue in \cite{liliu2018,li2021complete}, we expect that the convolution inverse of $\delta+\lambda a$ can be written as $\delta$ minus a completely monotone kernel. Hence, we expect that $r_{\lambda}$ is CM, which is much more than being nonnegative. 

\begin{proposition}\label{pro:cmfunc}
If $a$ is CM which is integrable on $(0,1)$ and not identically zero, then $r_{\lambda}$ is CM and is strictly positive. Moreover, $(\delta-r_{\lambda})*a=a-r_{\lambda}*a=r_{\lambda}/\lambda$ and 
\begin{gather}
s_{\lambda}=(\delta-r_{\lambda})*1=1-\int_0^t r_{\lambda}(\tau)d\tau
\end{gather}
 are both CM functions. Moreover, $1-\int_0^t r_{\lambda}(s)ds$ is strictly positive for all $t$.
\end{proposition}

To prove this, we need some auxilliary tools.  The discrete case in \cite{li2021complete} is proved based on the generating functions, so an analogue of the generating functions for the continuous complete monotone functions is needed here. In particular, we consider the following transform of $a$ if $\int_{0}^{\infty} 1\wedge t a(t)dt<\infty$:
\begin{gather}
F_a(z)=\int_{(0,\infty)}(1-e^{-z t})a(t)\,dt, \quad z\in \mathbb{C}.
\end{gather}
This is related to the so-called complete Bernstein functions (see \cite[Chap. 6]{ssv12}).
A function $f$ is said to be a complete Bernstein function, if there exists a complete monotone function $m$ with $\int_{0}^{\infty} (1\wedge t) m(t)dt<\infty$ such that 
\[
f(\lambda)=a+b\lambda+\int_{(0,\infty)}(1-e^{-\lambda t})m(t)\,dt, \quad a\ge 0, b\ge 0.
\]
Note that $a, b$ and $m(\cdot)$ are uniquely determined.   

The following characterization of the complete Bernstein function from \cite[Theorem 6.2]{ssv12} is useful.
\begin{lemma}\label{lmm:completebern}
Suppose that $f(\cdot)$ is nonnegative on $(0,\infty)$. Then, $f(\cdot)$ is a complete Bernstein function if and only if $f$ has an analytic continuation to $\mathbb{C}\setminus (-\infty, 0]$ such that $\mathrm{Im}(z)\cdot \mathrm{Im}(f(z))\ge 0$ and $f(0+)=\lim_{\lambda\in 0^+, \lambda\in \R}f(\lambda)$ exists.
\end{lemma}
We will now use this result to prove Proposition \ref{pro:cmfunc}.

\begin{proof}[Proof of Proposition \ref{pro:cmfunc}]
For the notational convenience, we will omit the dependence of $r$ on $\lambda$. Namely,
$r$ means $r_{\lambda}$.

Since $a$ is completely monotone, it is then nonincreasing. Moreover, by the assumption that
$a$ is integrable on $(0, 1)$, the regularized kernel 
\[
a_{\epsilon}(t):=a(t)e^{-\epsilon t}
\]
is integrable on $(0,\infty)$ and is also CM (see  \cite[Theorem 1.6]{ssv12} for the fact that the CM property is closed under multiplication). 
Hence,
\[
F_{a_{\e}}(z):=\int_{(0,\infty)}(1-e^{-z t})a_{\epsilon}(t)\,dt
\]
is complete Bernstein, and it is clearly nonnegative for $z>0$. 
Moreover, $F_{a_{\e}}$ is not a constant by the assumption on $a$. Its imaginary part is a harmonic function on $\mathbb{C}\setminus (-\infty, 0]$, nonnegative for $\mathrm{Re}(z)\ge 0$ by Lemma \ref{lmm:completebern}, and is zero on $(0,\infty)$. We thus infer that its imaginary part is strictly positive in the upper half plane and strictly negative in the lower half plane. 

Consider the resolvent of $a_{\e}$ by
\[
r_{\e}+\lambda r_{\e}*a_{\e}=\lambda a_{\e}.
\]
Here, $r_{\e}$ means $r_{\lambda, \e}$ and \cite[Lemma 2]{miller1968volterra} implies $r_{\e}\ge 0$.  Denote $m_{a_{\e}}:=\int_0^{\infty}a_{\e}(t)\,dt$. Direct computation gives
\[
F_{r_{\e}}(z)=\frac{\lambda F_{a_{\e}}(z)}{(1+\lambda m_{a_{\e}})(1+\lambda m_{a_{\e}}-\lambda 
F_{a_{\e}}(z))}.
\]
By the properties of $F_{a_{\e}}$, one finds that $F_{r_{\e}}(s)$ is nonnegative on $(0,\infty)$ and is analytic on $\mathbb{C}\setminus (-\infty, 0]$.
Moreover, 
\[
\mathrm{Im}(F_{r_{\e}}(z))=\frac{\lambda \mathrm{Im}(F_{a_{\e}}(z))}{|1+\lambda m_{a_{\e}}-\lambda F_{a_{\e}}(z)|^2}.
\]
Hence, $\mathrm{Im}(z)\cdot \mathrm{Im}(F_{r_{\e}}(z))\ge 0$ also holds. Moreover, $F_{r_{\e}}(0+)$ clearly exists. Hence, $F_{r_{\e}}(z)$ is also a complete Bernstein function. 
By the uniqueness of the representation of the complete Bernstein function, $r_{\e}(t)$ is CM by Lemma \ref{lmm:completebern}.
Moreover, by \cite[Corollary 1.7]{ssv12}, 
\[
r(t)=\lim_{\epsilon\to 0} r_{\e}(t)
\]
is completely monotone.

Since $m_{r}=(\lambda m_a)/(1+\lambda m_a)$ where $m_a$ could be $\infty$, $s_{\lambda}=1-\int_0^t r_{\lambda}(s)\,ds\ge 0$. Moreover, $s'_{\lambda}=-r_{\lambda}$, which is the negation of a complete monotone function, so $s_{\lambda}$ is CM. Since it is CM and not identically zero, it is strict positive by the Bernstein theorem.
\end{proof}

\section{Completely positive kernels on uniform meshes}\label{sec:cpuniform}

In this section, we first investigate the discrete analogue of the complete positivity for uniform meshes. 
We need the convolution on (uniform) discrete meshes.  The usual convolution is defined by
\begin{gather}
(a*b)_n=\sum_{j=0}^n a_{n-j}b_j.
\end{gather}
It is clear that this operation is commutative,and $\delta_d=(1, 0, 0, \cdots)$ is the convolution identity. The convolutional inverse of $a$ is the sequence $b$ satisfying $a*b=b*a=\delta_d$ and one may denote $a^{(-1)}:=b$. Clearly, $a^{(-1)}$ exists if and only if $a_0\neq 0$. The complementary kernel $a^c$ is the one satisfying
\begin{gather}
a*a^c=a^c*a=(1,1,\cdots).
\end{gather}
It is clear that $a^c=a^{(-1)}*(1,1,\cdots)$.

Similar to Definitions \ref{def:resol} and \ref{def:cpcont}, one may define the following.
\begin{definition}
A sequence $a=(a_0, a_1, \cdots)$ with $a_0\neq 0$ is said to be completely positive if the resolvent sequence given by 
\[
r_{\lambda}+\lambda r_{\lambda}*a=\lambda a
\]
is nonnegative for all $\lambda>0$ and it holds that $\sum_{i=0}^n (r_{\lambda})_i\le 1$ for all $n$.
\end{definition}

With the fact that the complementary kernel $a^c$ satisfies $a^c=a^{(-1)}*(1,1,\cdots)$ and motivated by Lemma \ref{lmm:contcp},  one naturally considers the following conditions for the inverse $b=a^{(-1)}$:
\begin{gather}\label{eq:signproperty}
b_0>0; \quad b_j\le 0, \quad  j\ge 1;  \quad \sum_{j=0}^n b_j \ge 0, n\ge 1.
\end{gather}
Similar to Lemma \ref{lmm:contcp}, one actually has
\begin{theorem}\label{thm:discretecpuni}
The sequence $a$ with $a_0\neq 0$ is completely positive if and only if 
the convolutional inverse $b=a^{(-1)}$ satisfies \eqref{eq:signproperty}.
\end{theorem}
\begin{proof}
Consider the ``$\Rightarrow$'' direction. By definition, one has $(r_{\lambda})_0=\lambda a_0/(1+\lambda a_0)$ which exists as long as $1+\lambda a_0\neq 0$. This clearly holds for $\lambda$ large enough since $a_0\neq 0$. Since $(r_{\lambda})_0\ge 0$, we infer that $a_0>0$ and thus $b_0=a_0^{-1}>0$.  

Since $a_0>0$, $r_{\lambda}$ is invertible for $\lambda>0$. Moreover, it holds that
\[
r_{\lambda}^{(-1)}=\delta_d+\lambda^{-1}a^{(-1)} \Longrightarrow r_{\lambda}=(\delta_d+\lambda^{-1}a^{(-1)})^{(-1)}.
\]
Then, in the elementwise limit sense, one has
\begin{gather}\label{eq:ainvtor}
a^{(-1)}=\lim_{\lambda\to\infty}\lambda(\delta_d-r_{\lambda}).
\end{gather}
Since $r_{\lambda}\ge 0$, one then finds that $b_j\le 0$ for $j\ge 1$. 

Since $\sum_{i=0}^n\lambda (\delta_d-r_{\lambda})_i\ge 0$ due to the completely positive requirement, then by \eqref{eq:ainvtor},
\[
\sum_{j=0}^n b_j=\lim_{\lambda\to\infty} \lambda \sum_{i=0}^n (\delta_d-r_{\lambda})_i \ge 0.
\]

For the ``$\Leftarrow$'' direction, since $a_0b_0=1$ and
\[
a_n b_0=-\sum_{j=1}^n a_{n-j}b_j, \quad n\ge 1.
\]
It is straightforward to see that $a_0>0$ and $a\ge 0$ by induction. 
Since $r_{\lambda}^{(-1)}=\delta_d+\lambda^{-1}a^{(-1)}$.then the first entry is 
positive and other entries are nonnegative as well. Similar argument shows that
 $r_{\lambda}\ge 0$.
 
Note that $a^c$ is nonnegative due to the third condition in \eqref{eq:signproperty}. 
By the fact that $a*a^c=(1,1,\cdots)$, one then has $b_0a_n\le 1=b_0a_0$. This implies that $a_n\le a_0$ for all $n\ge 1$  .
By the definition of $r_{\lambda}$, one then has for $n\ge m$ that
\[
\lambda \sum_{i=0}^na_i = \sum_{i=0}^n(r_{\lambda})_i
+\lambda \sum_{j=0}^n\left(\sum_{i=0}^{n-j}a_i \right)(r_{\lambda})_j
\ge \sum_{i=0}^m (r_{\lambda})_i \left(1+\lambda \sum_{i=0}^{n-m}a_i\right).
\]
If $\sum_{i=0}^{\infty}a_i<\infty$, $\sum_{i=n-m+1}^n a_i\to 0$. Otherwise, 
$\sum_{i=n-m+1}^n a_i \le m a_0$. In both cases, fixing $m$ and sending $n\to\infty$, one has $\sum_{i=n-m+1}^n a_i/\sum_{i=0}^{n-m}a_i \to 0$ and thus $\sum_{i=0}^m (r_{\lambda})_i\le 1$. Another way to see this is that $\delta_d-r_{\lambda}=\lambda^{-1}b*r_{\lambda}$ so that $1-\sum_{i=0}^n (r_{\lambda})_i=\lambda^{-1}r_{\lambda}*a^c$. The conclusion also follows.
\end{proof}

\section{Pseudo-convolution}\label{sec:pc}

To generalize the complete positivity to nonuniform meshes, we consider an operation which we call ``pseudo-convolution'' for two arbitrary 2D lower triangular arrays. This operation is motivated by the works by Liao. et al \cite{liao2020positive}, where they used the so-called ``discrete orthogonal convolution (DOC)'' kernels to perform some accurate analysis of ``backward differentiation formula'' (BDF) schemes and to investigate the positive-definiteness on nonuniform grids. A related tool is the ``discrete complementary convolution (DCC)'' kernel introduced in \cite{liao2018sharp,liao2019discrete}. These kernels turn out to be convenient tools to address the discretization of convolution operators on nonuniform grid. Here, we view this operation as a mapping that sends two array kernels into a new array kernel, which is an analogue of the standard convolution. Since this operation is very similar to the matrix multiplications, many properties about the DOC and DCC kernels can be understood naturally using this viewpoint.

We arrange the kernel $\{a_{n-j}^n\}$ into a lower triangular array $A$ of the following form
\begin{gather}\label{eq:arraykernel}
A=\begin{bmatrix}
a_0^{1} &  &  &  &  \\
a_1^{2}& a_0^{2} & & & \\
\cdots & \vdots & \vdots &  &  \\
a_{n-1}^{n} &  \cdots & a_1^{n} & a_0^{n} &\\
\cdots & \vdots & \vdots &   & \vdots\\
\end{bmatrix}.
\end{gather}
Denote $K$ to be the set of such kernels.

\begin{definition}
We define the pseudo-convolution $\pc: K\times K\to K$, $C=A\bar{*}B$,
by
\begin{gather}\label{eq:convnonuni}
c_{k}^n=\sum_{j=0}^k a_{k-j}^n b_j^{n+j-k}, \quad \text{or}\quad c_{n-k}^n=\sum_{j=k}^n a_{n-j}^n b_{j-k}^j.
\end{gather}
\end{definition}
The pseudo-convolution is illustrated in Figure \ref{fig:pconv}. 
To compute $c_2^4$, we take the subvector from $a_2^4$ to the rightmost in the row where $a_2^4$ lies, take the subvector from $b_2^4$ to the upmost element in the column where $b_2^4$ lies, and then take the dot product between these two subvectors, which will be $c_2^4$.
\begin{figure}[!htbp]
\centering  
\includegraphics[width=0.8\textwidth]{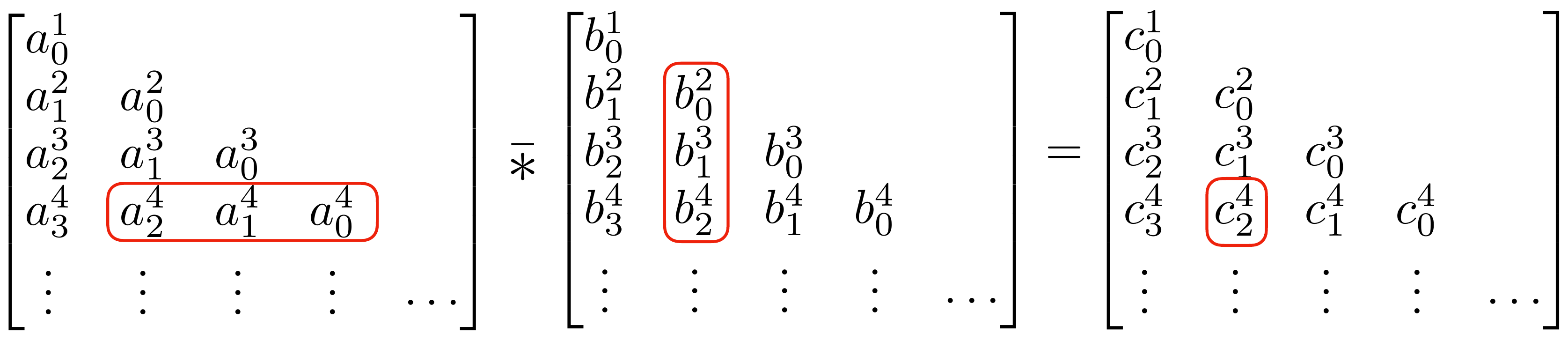}  
\caption{Illustration for the  pseudo-convolution.} 
 \label{fig:pconv} 
\end{figure}

We remark that such an array \eqref{eq:arraykernel} has been introduced already in \cite[Lemma 2.1]{liao2020positive} and the operation \eqref{eq:convnonuni} has appeared in  \cite{liao2019discrete,liao2021analysis} as well for the definition of the so-called DOC kernel there.  Moving one step further to make the operation for two arbitrary array kernels, it soon becomes a useful tool for nonuniform meshes.

\begin{remark}\label{rmk:conv1}
The pseudo-convolution here is defined for infinite arrays. By the definition, the convolution for $n\le N$ does not depend on the data with $n>N$. Hence, though the discussion here is for infinite arrays, the results can apply to array kernels with finite data.
\end{remark}

We introduce the following identity kernel
\begin{gather}\label{eq:I}
I=
\begin{bmatrix}
1 &  &    &  &\\
& 1 &  &  &\\
 &    & \vdots & & \\
& &    &1 & \\
 &  &    & & \vdots\\
\end{bmatrix}.
\end{gather}
Namely, $I_{n-k}^{n}=\delta_{nk}$.

This pseudo-convolution in general is not commutative. However, it has other desired basic properties as listed below. 
\begin{lemma}
The so-defined pseudo-convolution satisfies the following properties
\begin{enumerate}[(i)]
\item $(A+B)\pc C=A\pc C+B\pc C$, $A\pc(B+C)=A\pc B+A\pc C$; 
\item The associative law holds $(A\pc B)\pc C=A\pc(B\pc C)$;
\item $I\pc A=A$, $A\pc I=A$.
\end{enumerate}
\end{lemma}
\begin{proof}
We only have to verify the second property while the others are trivial. By definition
\begin{multline*}
[A\pc(B\pc C)]_{n-k}^n=\sum_{j=k}^n a_{n-j}^n (B\pc C)_{j-k}^j
=\sum_{j=k}^n a_{n-j}^n \sum_{\ell=k}^j b_{j-\ell}^j c_{\ell-k}^{\ell}\\
=\sum_{\ell=k}^n (\sum_{j=\ell}^n a_{n-j}^nb_{j-\ell}^j) c_{\ell-k}^{\ell}
=\sum_{\ell=k}^n (A\pc B)_{n-\ell}^n c_{\ell-k}^{\ell}=[(A\pc B)\pc C]_{n-k}^n.
\end{multline*}
\end{proof}

The following lemma explains why we call it pseudo-convolution. The verification is straightforward and we omit the proof.
\begin{lemma}
If $a_{j}^n=a_{j}$ and $b_{j}^n=b_j$ are both  independent of $n$, then  it reduces to the usual convolution.
\end{lemma}

Clearly, the kernel $I$ can be regarded as the identity.  Next, we introduce the inverse.
\begin{definition}
 $B$ is an inverse of $A$, if $A\pc B=I$.
\end{definition}

The following is a basic property regarding the inverse.
\begin{proposition}\label{pro:inv}
If $a_0^n\neq 0$ for all $n$, then $A$ has a unique inverse $B$ such that $A\pc B=I$. Moreover, it holds that $B\pc A=I$.
\end{proposition}

\begin{proof}
By the definition,
\begin{gather}
\sum_{j=k}^n a_{n-j}^n b_{j-k}^j=\delta_{nk}.
\end{gather}
This holds if and only if
\[
b_0^n=1/a_0^n, \quad b_{n-k}^n=-(a_0^n)^{-1}\sum_{j=k}^{n-1} a_{n-j}^n b_{j-k}^j, 1\le k \le n-1.
\]
Hence, $B$ is uniquely solved for each $n=1, 2, \cdots$. This verifies the first claim.

Now, we verify the second claim, namely
\begin{gather}
\sum_{j=k}^n b_{n-j}^n a_{j-k}^j=\delta_{nk}, \quad 1\le k \le n.
\end{gather}
For $n=1$, this holds clearly. We now do induction.
Suppose that this holds for $n \le m$ ($m \ge 1$). Consider $n=m+1$. 
Clearly, when $k=n$, this holds. For $1\le k \le n-1$,
\begin{gather}
\begin{split}
&\sum_{j=k}^n b_{n-j}^n a_{j-k}^j=b_0^na_{n-k}^n
+\sum_{j=k}^{n-1} a_{j-k}^j[-(a_0^n)^{-1}\sum_{\ell=j}^{n-1} a_{n-\ell}^n b_{\ell-j}^\ell] \\
&=b_0^na_{n-k}^n
-(a_0^n)^{-1} \sum_{\ell=k}^{n-1}a_{n-\ell}^n \sum_{j=k}^{\ell}  b_{\ell-j}^\ell
a_{j-k}^j =b_0^na_{n-k}^n
-(a_0^n)^{-1} \sum_{\ell=k}^{n-1}a_{n-\ell}^n\delta_{\ell k}=0.
\end{split}
\end{gather}
The second last equality is by induction hypothesis.
Hence, the desired claim holds by induction.
\end{proof}
The kernel $B$ is actually the ROC kernel defined in \cite{liao2021analysis}.  Since it is both left and right inverse, we will simply denote
\[
A^{(-1)}:=B, \quad \text{such that}~B\pc A=A\pc B=I.
\]

The following fact, though straightforward, is useful, which is reminiscent of the M-matrices (see \cite{plemmons1977m}).
\begin{lemma}\label{lmm:signofentry}
Let $B$ be the inverse of $A$. If $B$ has positive diagonal elements and nonpositive off-diagonal elements, then $A$ has nonnegative elements and the entries on the diagonal are positive. 
\end{lemma}
\begin{proof}
Let  $A=(a_{n-j}^n)$ and $B=A^{(-1)}=(b_{n-j}^n)$. 
Then, it is clear that $a_{0}^n=1/b_0^n>0$. 

For fixed $n$, suppose the claim is true for $j\ge k+1$ where $k\le n-1$.  
Then, for $j=k$, one has
\[
\sum_{j=k}^n a_{n-j}^n b_{j-k}^j=0.
\]
Then,
\[
a_{n-k}^nb_0^k=\sum_{j=k+1}^n a_{n-j}^n (-b_{j-k}^j)\ge 0,
\]
where $a_{n-j}^n\ge 0$ is due to induction hypothesis and $(-b_{j-k}^j)\ge 0$ for $j\ge k+1$ is due to the condition given.
 The claim is then proved.
\end{proof}

Next, we define the pseudo-convolution between a kernel and a vector.
We consider 
\[
V=\{x=(x_1, x_2,\cdots)^T:\quad x_i\in \R \}.
\]
 Define the pseudo-convolution $\pc$: $K\times V\to V$,
\begin{gather}
y=A\pc x
\end{gather}
by
\begin{gather}
y_{n}=\sum_{j=1}^{n}a_{n-j}^n x_{j}.
\end{gather}
\begin{remark}
Here, the index of $x$ starts with $i=1$ instead of $i=0$. This convention is adapted to the fact that there are only $n-1$ ``$a_{n-j}^n$'' elements for fixed $n$. This is also convenient for the implicit scheme \eqref{eq:integraldis}. If $a_{n-j}^n\equiv a_{n-j}$, we could understand $x$ as a kernel $x_j=b_{j-1}^{j}\equiv b_{j-1}$. Then $A\pc x$ reduces to the usual convolution. 
\end{remark}

It holds that
\begin{lemma}
$A\pc(B\pc x)=(A\pc B)\pc x$.
\end{lemma}
\begin{proof}
By the definition, one has
\[
\sum_{j=1}^n a_{n-j}^n (B\pc x)_{j}
=\sum_{j=1}^n a_{n-j}^n \sum_{\ell=1}^j b_{j-\ell}^j x_{\ell}
=\sum_{\ell=1}^n(\sum_{j=\ell}^n a_{n-j}^n b_{j-\ell}^j)x_{\ell}
=\sum_{\ell=1}^n(A\pc B)_{n-\ell}^nx_{\ell}.
\]
This then verifies the claim.
\end{proof}

\section{The complete positive kernels for nonuniform meshes}\label{sec:cpnonuni}

In this section, we explore the generalization of complete positivity to nonuniform meshes. 
The following kernel and its inverse will play important roles below.
\begin{gather}
L=
\begin{bmatrix}
1 &  &  &  &  \\
1& 1 & & &  \\
 \vdots & \vdots &  \vdots&  &  \\
1&1 & \cdots & 1 &  \\
 \vdots&\vdots  &\vdots  &\vdots  &\vdots  \\
\end{bmatrix},
\quad 
L^{(-1)}=
\begin{bmatrix}
1 &  &  &  &  \\
-1& 1 & & &  \\
 & -1 &  1&  &  \\
& & \vdots & 1 &  \\
&  &  &\cdots  &\vdots  \\
\end{bmatrix}.
\end{gather}
\begin{definition}
For a given $A$, the kernel $C_R$ with $A\pc C_R=L$ is called the right complementary kernel. The kernel $C_L$ with $C_L\pc A=L$ is called the left complementary kernel.
\end{definition}
The kernel $C_R$ is in fact the so-called ``right convolutional complementary'' (RCC) kernel in \cite{liao2023discrete} and $C_L$ is in fact the ``discrete convolutional complementary'' (DCC) kernel in \cite{liao2019discrete}.  

The following lemma is clear and we omit the proof.
\begin{lemma}\label{lmm:complementaryprop}
Let $A$ be a kernel that is invertible. Then,  $C_R=A^{(-1)}\pc L$ and $C_L=L\pc A^{(-1)}$.
\end{lemma}

Next, we consider the resolvent kernels for nonuniform meshes. 
\begin{lemma}\label{lmm:Rbasics}
Suppose the diagonal elements of $A$ are positive and its right complementary kernel is $C_R$. Then, the resolvent $R_{\lambda}$ defined by
\begin{gather}\label{eq:disreldef}
R_{\lambda}+\lambda R_{\lambda}\pc A=\lambda A
\Leftrightarrow
A-R_{\lambda}\pc A=\frac{1}{\lambda}R_{\lambda}
\end{gather}
 always exists for $\lambda>0$.
Moreover, the following holds:
\begin{enumerate}[(a)]
\item $R_{\lambda}\pc A=A\pc R_{\lambda}$, $R_{\lambda}\pc A^{(-1)}=A^{(-1)}\pc R_{\lambda}$;
\item $I-R_{\lambda}=(I+\lambda A)^{(-1)}=\lambda^{-1} R_{\lambda}\pc A^{(-1)}$;
\item The right complementary kernel of $R_{\lambda}$ is $\lambda^{-1}C_R+L$, namely
$R_{\lambda}\pc (\lambda^{-1}C_R+L)=L$.
\end{enumerate}
\end{lemma}
\begin{proof}
The relation \eqref{eq:disreldef} is equivalent to 
\[
(I-R_{\lambda})\pc (I+\lambda A)=I.
\]
The existence of $R_{\lambda}$ follows by the fact that the diagonal elements of $I+\lambda A$ are nonzero. 
Moreover, by Proposition \ref{pro:inv}, 
\[
 (I+\lambda A)\pc (I-R_{\lambda})=I.
\]
This then implies that $R_{\lambda}\pc A=A\pc R_{\lambda}$, which immediately implies that $R_{\lambda}\pc A^{(-1)}=A^{(-1)}\pc R_{\lambda}$.
The assertion in (b) is straightforward. 

For the last assertion, using the relation $R_{\lambda}+\lambda R_{\lambda}\pc A=\lambda A$, one has
\[
R_{\lambda}^{(-1)}= I+\lambda^{-1} A^{(-1)}.
\]
Convolving $L$ on the right gives the result. 
\end{proof}

The following describes the asymptotic behavior of the resolvents, which could be insightful. The intuition comes from the simple relation for real numbers $(1+\lambda a)^{-1}(\lambda a)=1-\lambda^{-1}a^{-1}+O(\lambda^{-2})$.
\begin{lemma}\label{lmm:resolvasym}
Suppose that $A$ is invertible. The resolvent $R_{\lambda}$ satisfies the following as $\lambda\to \infty$:
\[
R_{\lambda}=I-\lambda^{-1}A^{(-1)}+O(\lambda^{-2}).
\]
The $O(\lambda^{-2})$ is elementwise under the limit $\lambda\to +\infty$.
\end{lemma}
\begin{proof}
Let $N=R_{\lambda}-(I-\lambda^{-1}A^{(-1)})$. Then,
\[
I=(I-R_{\lambda})\pc (I+\lambda A)=(\lambda^{-1}A^{-1}-N)\pc (I+\lambda A).
\]
This gives
\[
N+\lambda^{-1}N\pc A^{(-1)}=\lambda^{-2}A^{(-1)}\pc A^{(-1)}.
\]
We can then solve elements of $N$ for $n=1,2, \cdots$ to see that 
each element is indeed $N_{n-j}^n=O(\lambda^{-2})$.
\end{proof}

Similar to the time continuous case and the case for uniform meshes, we define the following.
\begin{definition}
We say a kernel $A$ is a completely positive kernel if 
\begin{gather}\label{eq:Rproperty}
0<(R_{\lambda})_0^n<1, \quad  (R_{\lambda})_{n-j}^n\ge 0
\end{gather} 
and 
\begin{gather}
\sum_{j=1}^n (R_{\lambda})_{n-j}^n\le 1
\end{gather} 
for all $\lambda>0$. 
\end{definition}

Similar to Theorem \ref{thm:discretecpuni}, one has the following observation.
\begin{theorem}\label{thm:cpnonuniform}
An array kernel $A$ is a completely positive kernel if and only if its pseudo-convolutional inverse $B=A^{(-1)}=(b_{n-j}^n)$
satisfies the following conditions
\begin{gather}\label{eq:propertyB}
\begin{split}
& b_0^n>0, \quad b_{n-j}^n\le 0, \quad \forall j<n,n\ge 1\\
& \sum_{j=1}^n b_{n-j}^n \ge 0, \quad \forall n\ge 1.
\end{split}
\end{gather}
\end{theorem}
\begin{proof}
For the ``$\Leftarrow$'' direction , since $R_{\lambda}^{(-1)}= I+\lambda^{-1} A^{(-1)}$, and $b_{n-j}^n\le 0$ for $j<n$,
then $R_{\lambda}^{(-1)}$ has positve diagonal elements and nonpositive off-diagonal elements. 
By Lemma \ref{lmm:signofentry},  $R_{\lambda}$ is nonnegative. 

Moreover, since $I-R_{\lambda}=\lambda^{-1}R_{\lambda}\pc B$, one then finds that 
\[
L-R_{\lambda}\pc L=\lambda^{-1}R_{\lambda}\pc C_R.
\]
Note that $C_R$ has nonnegative entries by the property that $\sum_{j=1}^n b_{n-j}^n \ge 0$, which then implies that $L-R_{\lambda}\pc L$ has nonnegative entries, or in other words
\[
1-\sum_{j=1}^n (R_{\lambda})_{n-j}^n \ge 0.
\]

For the ``$\Rightarrow$'' direction, using Lemma \ref{lmm:resolvasym}, one finds that 
\[
B=\lim_{\lambda\to\infty} \lambda (I-R_{\lambda}).
\]
Hence, for $j<n$, one has 
\[
b_{n-j}^n=\lim_{\lambda\to \infty} -\lambda (R_{\lambda})_{n-j}^n \le 0.
\]
The fact $b_0^n>0$ is clear. 
Moreover, $C_R=A^{(-1)}\pc L=\lim_{\lambda\to\infty} \lambda (I-R_{\lambda})\pc L$, then $\sum_{j=1}^n (R_{\lambda})_{n-j}^n\le 1$ implies that the entries of $C_R$ are nonnegative, or
\[
\sum_{j=1}^n b_{n-j}^n \ge 0.
\]
\end{proof}

\section{Application to L1 scheme for fractional differential equations}\label{sec:application}

In this section, we look at one illustrating example for how the theory above could be used. In particular, we look at the L1 discretization on nonuniform meshes for the time fractional differential equations and establish a discrete analogue of 
\cite[Theorem 5]{clement1979abstract}. 

Consider the following time fractional differential equations
\begin{gather}\label{eq:fracode}
D_c^{\alpha}u \in -\cA(u), \quad u(0)=u_0,
\end{gather}
where $\cA: \cX\to \cX$ is m-accretive for some Banach space $\cX$. This means that
\begin{itemize}
\item for any $x_1, x_2\in \cX$, $y_1\in \cA(x_1)$ and $y_2\in \cA(x_2)$, one has
\[
\langle y_1-y_2, w\rangle\ge 0, \forall w\in J(x_1-x_2),
\]
where $J: \cX\to \cX'$ is the dual map ($\forall y\in J(x)$, $\|y\|=\|x\|$, $\langle x, y\rangle=\|x\|^2$).
\item $R(I+\cA)=\cX$ (the range of $I+\cA$ is full).
\end{itemize}
The Caputo derivative is defined by
\begin{gather}\label{eq:captra}
D_c^{\alpha}u=\frac{1}{\Gamma(1-\alpha)}\int_0^t\frac{u'(s)}{(t-s)^{\alpha}}ds.
\end{gather}
The fractional differential equation \eqref{eq:fracode} is equivalent to the integral equation (see \cite{diethelm10} and also \cite{liliu18frac,liliu2018compact} for generalized versions)
\begin{gather}\label{eq:fracint}
u_0\in u(t)+g_{\alpha}*\cA(u(\cdot)),
\end{gather}
where
\begin{gather}\label{eq:kernelfode}
g_{\alpha}(t):=\frac{1}{\Gamma(\alpha)}t_+^{\alpha-1}.
\end{gather}
Direct computation verifies that the kernel is completely monotone, and thus completely positive.

For discretization, suppose that the computational time interval is $[0, T]$. Let $0=t_0<t_1<t_2<\cdots<t_N=T$ be the grid points.  We define
\begin{gather}
\tau_n:=t_n-t_{n-1}, \quad n\ge 1.
\end{gather}
Let $u_n$ be the numerical solution at $t_n$ and denote
\[
\nabla_{\tau}u_n=u_{n}-u_{n-1},\quad n\ge 1.
\]
The L1 scheme \cite{lin2007finite,stynes2017error} can be reformulated as
\begin{gather}\label{eq:L1a}
D^{\alpha}u(t_n) \approx \cD_{\tau}^{\alpha}u_n:=C\pc \nabla_{\tau} u_n=C\pc L^{(-1)}\pc (u-u_0)_n,
\end{gather}
where
\begin{gather}\label{eq:L1b}
c_{n-j}^n=\frac{1}{\tau_j\Gamma(1-\alpha)}\int_{t_{j-1}}^{t_j}(t_n-s)^{-\alpha}\,ds.
\end{gather}
Thus, the discrete scheme is given by
\begin{gather}\label{eq:discrete}
\cD_{\tau}^{\alpha}u_n \in -\cA(u_n).
\end{gather}
It can then be verified easily that $B:=C*L^{(-1)}$ satisfies \eqref{eq:propertyB}. Hence, it corresponds to a completely positive kernel on nonuniform mesh $A=(a_{n-j}^n)=B^{(-1)}$.  In other words, \eqref{eq:discrete} can be converted into
\begin{gather}\label{eq:integraldis}
u_0\in u_n+\sum_{j=1}^n a_{n-j}^n \cA(u_j).
\end{gather}
Here, $a_{n-j}^n$ is like the inegral of $g_{1-\alpha}(t_n-s)$ on $(t_{j-1}, t_j)$.

Following the standard argument, both the solutions to \eqref{eq:fracode} and \eqref{eq:integraldis} can be 
approximated by the Yosida approximation.  In particular, 
\begin{gather}
J_{\lambda}:=(I+\lambda \cA)^{-1}, \quad \cA_{\lambda}=\lambda^{-1}(I-J_{\lambda}).
\end{gather}
We consider then
\begin{gather}\label{eq:integralaprox}
u_0^{\lambda}=u_n^{\lambda}+\sum_{j=1}^n a_{n-j}^n \cA_{\lambda}(u_j^{\lambda}).
\end{gather}

Now, we establish a discrete analogue of \cite[Theorem 5]{clement1979abstract}. In particular, let $P\subset \cX$
be a closed convex cone such that
\begin{gather}\label{eq:Ppreserving}
J_{\lambda}P\subset P, \quad \forall \lambda>0.
\end{gather}
One has the following
\begin{proposition}
Suppose that $u_0\in P$ and \eqref{eq:Ppreserving} holds. Moreover, if $A$ is a completely positive kernel
(or $B$ satisfies \eqref{eq:propertyB}), then $u_n^{\lambda}\in P$ for all $n\ge 0$ and $\lambda>0$.
Consequently, the numerical solution $u_n\in P$.
\end{proposition}
\begin{proof}
We only need to consider the approximation \eqref{eq:integralaprox} and show that $u_n^{\lambda}\in P$.
Rewrite \eqref{eq:integralaprox} as
\[
u_0^{\lambda}+\lambda^{-1}A\pc J_{\lambda}(u_n^{\lambda})
=u_n^{\lambda}+\lambda^{-1}A\pc u_n^{\lambda}.
\]
Taking pseudo-convolution on both sides with $I-R_{\lambda^{-1}}$, noting by Lemma \ref{lmm:Rbasics} that
\[
(I-R_{\lambda^{-1}})\pc (I+\lambda^{-1}A)=I,
\]
one has
\[
u_n^{\lambda}=(1-\sum_{j=1}^n (R_{\lambda^{-1}})_{n-j}^n) u_0^{\lambda}+R_{\lambda^{-1}}\pc J_{\lambda}(u_n^{\lambda}).
\]
Hence,
\[
\left(u_n^{\lambda}-(R_{\lambda^{-1}})_0^n J_{\lambda}(u_n^{\lambda})\right)
=\left(1-\sum_{j=1}^n (R_{\lambda^{-1}})_{n-j}^n\right) u_0^{\lambda}+\sum_{j=1}^{n-1}(R_{\lambda^{-1}})_{n-j}^n J_{\lambda}(u_j^{\lambda}).
\]

By induction, if $u_j^{\lambda}\in P$ for all $j<n$, then Theorem \ref{thm:cpnonuniform} implies that all the coefficients on the right hand side are nonnegative and $(R_{\lambda^{-1}})_0^n\in (0, 1)$. Consequently,
\[
\left(u_n^{\lambda}-(R_{\lambda^{-1}})_0^n J_{\lambda}(u_n^{\lambda})\right)=f \in P.
\]
To see that $u_n^{\lambda}\in P$, we consider the iteration
\[
w^{k+1}=f+(R_{\lambda^{-1}})_0^n J_{\lambda}(w^k), \quad w^0=u_{n-1}^{\lambda}.
\]
Since $(R_{\lambda^{-1}})_0^n\in (0, 1)$ and $J_{\lambda}$ is a contraction, the contraction mapping theorem ensures that $w^k$ converges to the unique fixed point $w_*$ and this must be $u_n^{\lambda}$.
On the other hand, for each $k$, $w^k\in P$ is clear so $u_n^{\lambda}=w_*\in P$
follows since $P$ is closed. The proof is then complete.
\end{proof}

As an application, one may consider the example modified from \cite[Example 1]{clement1979abstract}. Let $\Omega$
be a bounded domain with smooth boundary. Consider
\[
D_c^{\alpha}u= \Delta u\tcb{-}\beta(u).
\]
Here, $\beta(0)=0$ and $\beta(u)=F'(u)$ for some lower semi-continuous, proper convex function $F$. Suppose the initial data $u_0\in W_0^{1,2}(\Omega)$, $u_0\ge 0$ and
\[
\int_{\Omega} F(u_0(x))\,dx<\infty.
\]
Then, one can take $P=L_+^2(\Omega)$ (the set of square integrable functions that are nonnegative). Then,
all the assumptions hold so that the numerical solution to the L1 scheme satisfies $u_n \in L_+^2(\Omega)$. In other words, the numerial solution is nonnegative.

\section*{Acknowledgement}

This work was financially supported by the National Key R\&D Program of China, Project Number 2021YFA1002800 and 2020YFA0712000. The work of Y. Feng was partially sponsored by NSFC 12301283, Shanghai Sailing program 23YF1410300 and Science and Technology Commission of Shanghai Municipality (No. 22DZ2229014). The work of L. Li was partially supported by NSFC 12371400 and 12031013,  Shanghai Science and Technology Commission (Grant No. 21JC1403700, 20JC144100), the Strategic Priority Research Program of Chinese Academy of Sciences, Grant No. XDA25010403.

\bibliographystyle{plain}
\bibliography{frac}

\end{document}